%% file: main.tex
\documentclass[10pt,a4paper]{amsart}
\usepackage[utf8]{inputenc}
\usepackage{amsmath}
\usepackage{amsfonts}
\usepackage{amssymb}
\usepackage{graphicx}
\author{Benjamin Engel}
\email[Benjamin Engel]{engel.benjamin@gmail.com}
\title{Log-Concavity of the Overpartition Function}

\usepackage{hyperref}

\include{command}

\graphicspath{{Pictures/}} 
\usepackage{float}
\usepackage{amsthm,amssymb,amstext,amscd,amsfonts,amsbsy,amsxtra,latexsym,amsmath,mathtools}

\usepackage{cleveref} 
\crefname{thm}{theorem}{theorems}
\crefname{equation}{eq.}{eqs.}
\crefname{lemma}{lemma}{lemmas}
\crefname{figure}{figure}{figures}

\newcommand{\re}[1]{(\ref{#1})}

\usepackage[english]{babel}
\usepackage{scrextend}
\usepackage[square, numbers, comma, sort&compress]{natbib} 
\usepackage{enumerate}

\usepackage[]{todonotes}   

\theoremstyle{plain}

\newtheorem{thm}{Theorem}[section]

\newtheorem{lemma}[thm]{Lemma}

\theoremstyle{definition}
\newtheorem*{nonthm}{Theorem}
\newtheorem*{defi}{Definition}
\theoremstyle{remark}

\newcommand{\hl}[1]{\textit{\textbf{#1}}}

\usepackage{tabu}

 \thanks{The author is grateful for useful advice and guidance from Professor Bringmann, Dr. Krauel, Dr. Li, Dr. Mertens and  Dr. Rolen.} 

\begin{document}

\begin{abstract}
	 We prove that the overpartition function $  \op $ is log-concave for all $ n\ge 2 $. The proof is based on Sills Rademacher type series for $ \op $ and inspired by DeSalvo and Pak's proof for the partition function.
\end{abstract}

\maketitle
\noindent \textbf{MSC Keywords.} 05A17, 11P82, 11F20, 11F37

\section{Introduction and Statement of Results}
\input{Chapter1}

\input{Chapter3}

\nocite{*}
\bibliographystyle{amsplain}
\bibliography{main}

\end{document}

%% file: command.tex
\newcommand{\N}{\mathbb{N}}

\newcommand{\R}{\mathbb{R}}
\newcommand{\Z}{\mathbb{Z}}
\renewcommand{\ref}[1]{(\ref{#1})}
\newcommand{\op}{\overline{p}(n)}

\newcommand{\hm}{\hat{\mu}}

\newcommand{\EopA}{\sum_{0\le h < k \atop (h,k)=1} \limits \f{\omega(h,k)^2}{\omega(2h,k)}e\left(-\frac{nh}{k}\right)}
\newcommand{\Eop}{\f{1}{2 \pi} \sum_{k \ge 1 \atop 2 \nmid k} \sqrt{k} \EopA \f{d}{dn}\left( \f{\sinh(\pi \sqrt{n} /k)}{\sqrt{n}}\right) }

\newcommand{\ask}{A^*_k(n)}

\newcommand{\f}[2]{\frac{#1}{#2}}

\newcommand{\ohk}{\eexp{\pi i \sum\limits_{r=1}^{k-1} \frac{r}{k}\left(\frac{hr}{k}-\left\lfloor \frac{hr}{k}\right\rfloor - \frac{1}{2} \right)}}
\newcommand{\Eask}{\sqrt{k}\sum\limits_{(h,k)=1 \atop h\in \left( \Z / k\Z \right) ^\times} \omega(h,k)e\left( -\frac{hn}{k}\right) }

\newcommand{\eexp}[1]{\exp\left( #1 \right)}

%% file: Chapter1.tex
\noindent A \textbf{partition} of a positive integer $n$ is a non-increasing sequence of positive integers $a_1,\dots,a_l$  whose sum is equal to $n$. The number of partitions of a positive integer is commonly denoted by $p(n)$.
For example, the partitions of 4 are $$\lbrace 1,1,1,1\rbrace, \lbrace 1,1,2\rbrace, \lbrace 2,2\rbrace , \lbrace 1,3 \rbrace, \ \lbrace 4 \rbrace.$$
 Determing this number might seem like a very simple objective, a priori. However, in order to prove strong results about $p(n)$, one needs significant amounts of theory
 such as complex analysis, the theory of modular forms, and knowledge about Kloosterman sums and Bessel functions. For example,  with the circle method, one can obtain series representations or investigate the asymptotic behavior of the partition function, as well as prove other results.

Much of the modern study of properties of $p(n)$ begins with Ramanujan and Hardy, who found among many other results the asymptotic behavior of $p(n)$:
\begin{equation}\label{p:asymp}
	p(n) \sim \frac{\eexp{\pi \sqrt{\frac{2n}{3}}}}{4n\sqrt{3}}.
\end{equation} 
In 1917, Littlewood, Ramanujan and Hardy invented in 
\cite{hardy1918asymptotic} the circle method to obtain  \re{p:asymp}.
\\

Based on their results, Rademacher extended \re{p:asymp} to an exact formula in 1937 \cite{rademacher1938partition}. In the same year, D. H. Lehmer published  a formula with an error term \cite{l1}:
\begin{equation}\label{p:lehmer-rademacher}
	p(n)= \frac{\sqrt{12}}{24n-1}\sum\limits_{k=1}^{N}\ask
	\left[\left(1-\frac{k}{\mu}\right)
	e^{\mu/k} +
	\left(1+\frac{k}{\mu} \right) e^{-\mu /k} \right] +\dot{R}_2(n,N) .
\end{equation}
Here 
$$\mu := \mu(n)=\frac{\pi}{6}\sqrt{24n-1}$$ 
and 
\begin{equation*}\label{p:Ak}
	A_k(n):=\sqrt{k}\ask =\Eask,
\end{equation*}
with
$$e(x):=\eexp{2\pi i x}.$$
The term $\dot{R}_2/(n,N)$ is Lehmers reminder Term.
Finally, 
\begin{equation*}\label{ohk}
	\omega(h,k):=\ohk
\end{equation*}
for positive integers $h$ and $k$.

 We come now to the main result of this paper. Therefore, we recall what it means for  a function to be log-concave.
For a real valued function $f$ and positive integer $n$, define $\mathcal{C}(f(n))$ by 
\begin{equation*}\label{def:C}
	\mathcal{C}(f(n)):=\log(f(n+1))-2\log(f(n))+\log(f(n-1)).
\end{equation*}
\begin{defi}
	A function $f$ is \hl{log-concave} for a positive integer $n$ if $\mathcal{C}(f(n))\le 0 $.
\end{defi}
In 2013, Desalvo and Pak utilize \re{p:lehmer-rademacher} in \cite{main} to develop the following result about the log-concavity of $p(n)$:
\begin{nonthm}[Desalvo and Pak]\label{DP:thm}
	The sequence $p(n)$ is log-concave for all $n >25$.
\end{nonthm} 
They show this by reorganizing $p(n)$ into a main part $T(n)$
and an error term $R(n)$.
Then they use \re{DP:lemma} to get an upper and lower bound for   $\mathcal{C}(T(n))$ which leads to upper and lower bound for $\mathcal{C} (p(n)) $. In this paper we will derive a similar Statement for overpartitions by a careful analysis of the Rademacher-type series given by Sills\cite{s1}.
\begin{lemma}[DeSalvo and Pak]\label{DP:lemma}
	Suppose $f(x)$ is a positive, increasing function with two continuous derivates for all $x>0$, and that $f'(x)>0$ and decreasing and $f''(x)<0$ is increasing for all $x\ge1$. Then for all $ x>1$
	$$f''(x-1) <f(x+1)-2f(x)+f(x-1)<f''(x+1) .$$
\end{lemma}
In this paper we extend this result to overpartitions.  
  An \textbf{overpartition} of a positive integer $n$  is a partition of $n$ in which the last occurrence of a number can be distinguished, which we do by overlining it. For example, the overpartitions of 4 are 
\begin{align*}
	&\lbrace 1,1,1,1\rbrace, \lbrace 1,1,1,\overline{1}\rbrace, \lbrace 1,1,2\rbrace , \lbrace 1,1,\overline{2}\rbrace, \lbrace 1,\overline{1},2\rbrace, \lbrace 1,\overline{1},\overline{2}\rbrace, \lbrace 2,2\rbrace, \lbrace 2,\overline{2}\rbrace, \lbrace  1,3\rbrace, \lbrace 1 ,\overline{3}\rbrace,\\
	& \lbrace \overline{1},3 \rbrace, \lbrace \overline{1},\overline{3}\rbrace, \lbrace 4 \rbrace \text{ and } \lbrace \overline{4} \rbrace.
\end{align*}
 Commonly, the number of overpartitions of a positive integer $n$ is denoted by $\overline{p}(n)$.
 Sills \cite{s1} rediscovered Zuckermann's \cite{zuckerman} formula for the overpartition and pointed out that it is indeed a Rademacher-type series

 \begin{equation}\label{op:rademacher}
 \op = \Eop,
 \end{equation}
 by using  the generating function for the overpartition
 \begin{equation}\label{genfnc:op}
 \sum\limits_{n\ge1}\op q^n =\prod\limits_{n\ge 1}\frac{(1+q^n)}{(1-q^n)},
 \end{equation}
 and the circle method.

In this paper, our main result is the following.
\begin{thm}\label{main1}
	The function $\op $ is log-concave for $n\ge 2$.
\end{thm}

\Cref{t1} illustrates the \Cref{main1} and the behavior of $\op$ and $p(n)$ for increasing $n$.\\
\begin{table}[h!]
	\caption{Comparison of $p(n)$ and $\op $ and their image under $\mathcal{C}$.}
	\begin{center}
	\label{t1}
		$\begin{array}{|c|c|c|c|c|}
		\hline
		n & p(n) & \op & \mathcal{C}(p(n)) & \mathcal{C}(\overline{p}(n))\\
		\hline
		\hline
		2 & 2 & 4 & -0.287682 & 0 \\ \hline
		3 & 3 & 8 & 0.105361 & -0.133531 \\ \hline
		4 & 5 & 14 & -0.174353 & -0.0206193 \\ \hline
		5 & 7 & 24 & 0.115513 & -0.0281709 \\ \hline
		6 & 11 & 40 & -0.14183 & -0.040822 \\ \hline
		7 & 15 & 64 & 0.0728373 & -0.0237165 \\ \hline
		8 & 22 & 100 & -0.0728373 & -0.0145047 \\ \hline
		9 & 30 & 154 & 0.0263173 & -0.0219976 \\ \hline
		10 & 42 & 232 & -0.0487902 & -0.0158805 \\ \hline
		15 & 176 & 1472 & 0.0067245 & -0.0103469 \\ \hline
		20 & 627 & 7336 & -0.0129263 & -0.0062498 \\ \hline
		25 & 1958 & 31066 & 0.000765534 & -0.0047666 \\ \hline
		30 & 5604 & 116624 & -0.00546266 & -0.0037539 \\ \hline
		35 & 14883 & 398640 & -0.000934141 & -0.00300351 \\ \hline
		40 & 37338 & 1263272 & -0.00273332 & -0.00250508 \\ \hline
		45 & 89134 & 3759240 & -0.00120242 & -0.00212943 \\ \hline
		50 & 204226 & 10605564 & -0.00173137 & -0.0018351 \\ 
		\hline 
		\end{array}$
	\end{center}
\end{table}
This paper is a condensed version of my Diploma thesis at the University of Cologne in 2014 which was supervised by Professor K. Bringmann and Dr. L. Rolen and is structured as follows. 
 First, we elaborate on an error bound for $\op$, and then we split $\op $ into the main term $\widehat{T}(n)$ and an reminder term $\widehat{R}(n)$ and show that $\widehat{T}(n)$ is itself log-concave.
Finally, we deduce the log-concavity of $\op $ from the log-concavity of $\widehat{T}(n)$ and bounding $\widehat{R}(n)$.

%% file: Chapter3.tex
\section{The error term of $\overline{p}(n)$}
In this section, we provide an analoguous error term for the overpartition function as Lehmer did in \cite{l1} for the partition function, which is labeled $R_2(n,N)$ in this paper (see (\ref{p:lehmer-rademacher})). \\
Our first step is to calculate  the derivative in \re{op:rademacher}: 
$$\frac{d}{dn}\left(\frac{\sinh(\hm/k)}{\sqrt{n}}\right) = 
\frac{\pi}{2kn}\left(
\cosh\left(\frac{\hm}{k}\right)-\frac{k}{\hm} \sinh\left(\frac{\hm}{k}\right)
\right), $$
where $\hm := \hm(n)=\pi\sqrt{n}$.
We then let
$$\widehat{A}_k(n):= \sum_{0 \le h<k \atop (h,k)=1}\limits\frac{\omega(h,k)^2}{\omega(2h,k)}e(-nh/k),$$
 and split up (\ref{op:rademacher}) so that for any integer $N\ge 1$,
\begin{align}
\op =&\frac{1}{2\pi}\sum_{k \ge 1 \atop 2 \nmid k }\sqrt{k}\widehat{A}_k(n)\frac{d}{dn}\left(\frac{\sinh(\hm/k)}{\sqrt{n}}\right) \nonumber \\
=& \frac{1}{2\pi}\sum_{k \ge 1 \atop 2 \nmid k }^N \sqrt{k}\widehat{A}_k(n)\frac{d}{dn}\left(\frac{\sinh(\hm/k)}{\sqrt{n}}\right) 
 + 
\underbrace{\frac{1}{4n}\sum_{k \ge N+1 \atop 2 \nmid k }\frac{\widehat{A}_k(n)}{\sqrt{k}}\left(
\cosh\left(\frac{\hm}{k}\right)-\frac{k}{\hm} \sinh\left(\frac{\hm}{k}\right)
\right)}_{=:\widehat{R}_2(n,N)}.\label{op:restterm}
\end{align}
To get a useful error bound, we want to estimate  $\left|R_2(n,N)\right|$ by comparing with an elementary function.
This is done in two steps. First, we state two lemmas, beginning with the following one. For now we set $f_n(k):=\cosh\left(\frac{\hm}{k}\right)-\frac{k}{\hm}\sinh \left(\frac{\hm}{k}\right)$ and prove the following Lemma.
\begin{lemma}\label{lemma_positive}
We have $f_n(k) \ge 0$ for all $k\in \R \backslash \lbrace 0 \rbrace$.
\end{lemma}
The previous Lemma follows immediately by rewriting $ \cosh(x) $ and $ \sinh(x) $ in there series representation.
The next step is to take a closer look at
 $\frac{\omega(h,k)^2}{\omega(2h,k)}e(-nh/k)$. To this end, we use that $ \omega(h,k) $ is a certain 24th root of unity,
and the fact that $\left| \frac{\omega(h,k)^2}{\omega(2h,k)}e(-nh/k)\right|=1$.  
Thus, we can trivially bound $\left|\widehat{A}_k(n) \right|$ from above by $k$. Using this and \Cref{lemma_positive}, we have
$$\left|\sum_{k\ge N+1} \frac{\widehat{A}_k(n)}{\sqrt{k}}f_n(k)\right|\le \sum_{k\ge N+1} |\sqrt{k}f_n(k)| = \sum_{k\ge N+1} \sqrt{k} f_n(k)=:S.$$
To simplify further, we need the absolute convergence of $S$, which is true if $S$ is convergent because of \Cref{lemma_positive}.\\

\begin{lemma}\label{konv}
The sum $S$ is convergent.
\end{lemma}
It follows from \Cref{lemma_positive,konv}  that we can rearrange  terms and obtain
\begin{equation}\label{vertauscht}
\sum_{k\ge N+1 \atop 2\nmid k }\sum_{m\ge 1}\frac{\sqrt{k}
\left(\frac{\hm}{k}\right)^{2m}(2m)
}{(2m+1)!} =
\sum_{m\ge 1}\sum_{k\ge N+1 \atop 2\nmid k }\frac{
\sqrt{k}\left(\frac{\hm}{k}\right)^{2m}(2m)
}{(2m+1)!}.
\end{equation}
Finally, we put this all together and prove the  following theorem.
\begin{thm}
Let $\op $ be defined as in \re{op:rademacher}.
 Then 
  $$\op = \frac{1}{2\pi}\sum_{k \ge 1 \atop 2 \nmid k }^N\limits \sqrt{k}\widehat{A}_k(n)\frac{d}{dn}\left(\frac{\sinh(\hm/k)}{\sqrt{n}}\right) +R_2(n,N),$$ where $\left|R_2(n,N)\right| \le \frac{N^{5/2}}{n\hm}\sinh\left(\frac{\hm}{N}\right).$
\end{thm}
\begin{proof}
We have already seen that $\left|R_2(n,N)\right| \le S$. Using \re{vertauscht} and the fact
  $\frac{1}{k^\alpha}$ is monotonically decreasing and positive in $[N,\infty)$ for an $ \alpha >0 $, we obtain 
\begin{equation*}
\left| \sum_{m\ge 1}\sum _{k\ge N+1 \atop 2 \nmid k} \frac{\sqrt{k}\left(\frac{\hm}{k}\right)^{2m} (2m)}{(2m+1)!} \right| \le \left| \sum_{m\ge 1}\int\limits_{N}^{\infty} \frac{\sqrt{x}\left(\frac{\hm}{x}\right)^{2m} (2m)}{(2m+1)!}  \mathrm{d}x\right|.
\end{equation*}
Since $ \frac{\sqrt{x}\left(\frac{\hm}{x}\right)^{2m} 2m}{(2m+1)!}$ is decreasing and continuous  for $x \in [N,\infty)$, we find
\begin{equation*}
\int\limits_{N}^{\infty} \frac{\sqrt{x}\left(\frac{\hm}{x}\right)^{2m} (2m)}{(2m+1)!}  \mathrm{d}x= \left[ \frac{-4m \hm^{2m} x^{3/2-2m}}{(4m-3)(2m+1)!} \right]_N^\infty.
\end{equation*}
For $m>\frac{3}{4} $, we find
\begin{equation*}
\lim\limits_{x\rightarrow \infty} \frac{-4m \hm^{2m} x^{3/2-2m}}{(4m-3)(2m+1)!} = 0.
\end{equation*}
Finally, we have
\begin{align*}
\left|R_2(n,N)\right| &\le\frac{1}{4n}  \sum\limits_{m\ge 1}\frac{4m}{4m-3}\frac{\left( \frac{\hm}{N}\right)^{2m} }{(2m+1)!} \frac{N^{3/2}}{4n} \le \frac{N^{3/2}}{n} \sum\limits_{m\ge 0} \frac{\left( \frac{\hm}{N}\right)^{2m} }{(2m+1)!} \\
&= \frac{N^{5/2}}{n\hm}\sinh\left( \frac{\hm}{N}\right) ,
\end{align*}
which completes the proof.
\end{proof}
\section{Bounding the main term of $\op$}
This section provides the final expression of $\op$ as the sum of two functions $\widehat{T}(n)$ and $\widehat{R}(n)$  and concludes with a proof of the log-concavity of $\widehat{T}(n)$.
In order to proceed, we now evaluate $\frac{d}{dn}\left(
          \frac{\sinh\left(\frac{\hm}{k}\right)}{\sqrt{n}}\right)$ in another way by recalling the identity $\sinh\left(z\right) = \frac{e^z - e^{-z}}{2}$. This shows that
          \begin{align}
           \frac{d}{dn}\left(\frac{\sinh\left(\frac{\hm}{k}\right)}{\sqrt{n}}\right) & =  \frac{d}{dn}\left(\frac{e^{\frac{\hm}{k}}-e^{\frac{-\hm}{k}}}{2\sqrt{n}}\right)
           = \frac{\hm'}{2k\sqrt{n}}\left(e^{\frac{\hm}{k}}+e^{\frac{-\hm}{k}}\right)-\frac{\left(e^{\frac{\hm}{k}}-e^{\frac{-\hm}{k}}\right)}{4n^\frac{3}{2}}.\label{op:err:1}
           \end{align}
           If we now use $\frac{d}{dn}\hm= \frac{\pi}{2 \sqrt{n}}$, then \re{op:err:1} simplifies to 
           \begin{align}
           \frac{\pi}{4nk}\left(e^{\frac{\hm}{k}}+e^{\frac{-\hm}{k}}\right)-\frac{1}{4n^\frac{3}{2}}\left(e^{\frac{\hm}{k}}-e^{\frac{-\hm}{k}}\right) 
           = \frac{\pi}{4n}\left(\frac{1}{k}\left(e^{\frac{\hm}{k}}+e^{\frac{-\hm}{k}}\right)- \frac{1}{\hm}\left(e^{\frac{\hm}{k}}-e^{\frac{-\hm}{k}}\right)\right) \nonumber \\
             =  \frac{\pi}{4nk}\left(\left(e^{\frac{\hm}{k}}+e^{\frac{-\hm}{k}}\right)- \frac{k}{\hm}\left(e^{\frac{\hm}{k}}-e^{\frac{-\hm}{k}}\right)\right) \notag \\
             = \frac{\pi }{4nk}\left(
            \left(1+ \frac{k}{\hm}\right)e^\frac{-\hm}{k}
            +
            \left(1-\frac{k}{\hm}\right)e^\frac{\hm}{k}
            \right).\label{op:errortermableitung}
            \end{align}
            Substituting \re{op:errortermableitung} into \re{op:restterm} gives us
             \begin{align*}\label{opparts}
              \overline{p}(n) &= \frac{1}{8n}\sum\limits_{k= 1 \atop (k,2)=1}^{N}\frac{\widehat{A}_k(n)}{\sqrt{k}}\left(
                 \left(1+ \frac{k}{\hm}\right)e^\frac{-\hm}{k}
                 +
                 \left(1-\frac{k}{\hm}\right)e^\frac{\hm}{k}
                 \right) +R_2(n,N)\\
                 &=  \widehat{T}(n) + \widehat{R}(n),\notag
              \end{align*}
              where 
              $$\widehat{A}_k(n):= \sum\limits_{0\le h<k \atop (h,k)=1}\frac{\omega(h,k)^2}{\omega(2h,k)}e\left(-\frac{nh}{k}\right).$$
              Here, $\widehat{T}(n)$ and $\widehat{R}(n)$ are defined as
              \begin{align*}
              \widehat{T}(n)&:= \frac{1}{8n}\left\lbrace  e^{-\hm}+\left(1-\frac{1}{\hm}\right)e^{\hm} \right\rbrace ,\\
              \widehat{R}(n)&:=\frac{e^{-\hm}}{8n\hm} + R_2(n,3).
              \end{align*}
              
              In order to show the log-concavity of $\widehat{T}(n)$, we use \Cref{DP:lemma} and set 
              $$f(x):= \log(\widehat{T}(x)).$$
              First, we show that $f(x)$ fulfills the necessary preliminaries of \Cref{DP:lemma}.  
              \begin{lemma}\label{f:1}
              The function $f(x)$ satisfies the following properties: 
                            \begin{enumerate}[(i)]
                            	\item If $x \ge 1$, then  $f(x)$ is increasing and positive. \label{lem:I}
                            	\item For all $x>1$, the function $f'(x)$ is positive.\label{lem:II}
                            	\item If $x\ge 1$, then $f'(x)$ is decreasing. \label{lem:III}
                            	\item For all $x\ge1$, the function $f''(x)$ is negative. \label{lem:IV}
                            	\item If $x\ge 3$, then $f''(x)$ is increasing.\label{lem:V}
                            \end{enumerate}
              \end{lemma}
              \begin{proof}
              	\re{lem:I}+\re{lem:II}:
                             We notice that 
                            \begin{equation}\label[equation]{OP:Tupp}
                            \widehat{T}(x)>\frac{1}{2}e^{\hm(x)}>1 \text{ for all }x>0,
                            \end{equation}
                            and hence
                            $$f(x)>\hm(x) - \log(2)>0 \text{ for all }x>1>\left(\frac{\log(2)}{\pi}\right)^2,$$
                            with $\hm(x)= \pi \sqrt{x}$.
                            This establishes positivity.
                            
                            As the composition of a couple twice differentiable function, $f(x)$ is also twice differentiable for all $x>0$.
                            
                           We show now that $f(x)$ is increasing by proving that $f'(x)>0$ for all $x \ge 1$.
                            By simple calculus, we obtain 
                                          \begin{equation}\label{borris}
                                          f'(x)=\frac{e^{2\hm(x)} -2\hm(x)^2}{2x\left(\hm(x) + e^{2\hm(x)}(\hm(x) -1)\right)}.
                                          \end{equation}
                                          Next we notice that for $x\ge 1>(1/\pi)^2$, 
                                          \begin{equation}\label{hm>1}
                                          	\hm(x)>1
                                          \end{equation}
                                          and so 
                                          $$2x\left(\hm(x) + e^{2\hm(x)}(\hm(x) -1)\right) >0.$$
                                          To show the denominator of \re{borris} is positive, we notice 
                                          $$\eexp{2\hm(x)} = \sum\limits_{n\ge 0}\frac{(2\hm(x))^n}{n!}> 1+2\hm(x)+2\hm(x)^2$$
                                          and obtain
                                          $$\eexp{3\hm(x)}-2\hm(x)^2\ge 1+2\hm(x)+2\hm(x)^2-2\hm(x)^2>0$$
                                          for all $x\in \R$, and so we have that $f(x)$ is increasing for $x\ge 1$.\\

	\re{lem:III}+\re{lem:IV}:
              We show that $f'(x)$ is decreasing by observing that $f''(x)<0$ for all $x \ge 1$.
              We calculate $f''(x)$ directly as
                            \begin{equation*}
                            f''(x)=\frac{
                            2\hm(x)^3+e^{4\hm(x)}(2-3\hm(x))+e^{2\hm(x)}\hm(x)(-3+2\hm(x)-2\hm(x)^2+4\hm(x)^3)
                            }{
                            4\left(x^{3/2}\pi +e^{2\hm(x)}x (\hm(x)-1)\right)^2
                            }.
                            \end{equation*}
                           Note that the denominator is positive for all $x\in \R^+$. 
                            Next, we take into account that 
                            \begin{equation}
                            2-3\hm(x) < -2\hm(x) < 0 \text{ for } x\ge 1 > \left( \frac{2}{\pi}\right) ^2. \label{OP:low1}
                           \end{equation}
                            Using \re{hm>1}, we obtain 
                      \begin{align*}
                      	 -3+2\hm(x)-2\hm(x)^2+4\hm(x)^3  < 4\hm(x)^3 -3 \text{ for } x\ge 1> (1/\pi)^2,
                      \end{align*}
                      which brings us to 
                      \begin{equation}
                      	0\le -3+4\hm(x)^3 \le 5\hm(x) \label{OP:low2}
                      \end{equation}
                      for $x\ge 1 > \left( \frac{3}{4}\right) ^{2/3}\frac{1}{\pi ^2}$.
                            With \re{OP:low1} and \re{OP:low2}, we have
                            \begin{align}\label{OP:upper}
                           &  2\hm(x)^3+e^{4\hm(x)}(2-3\hm(x))+e^{2\hm(x)}\hm(x)(-3+2\hm(x)-2\hm(x)^2+4\hm(x)^3)\\ 
                           \le&   2\hm(x)^3-2\hm(x) e^{4\hm(x)}+5\hm(x)^6 e^{2\hm(x)} \notag 
                           \le  0  \notag,
                            \end{align}
                            for $x \ge 1 $.
                            Therefore, we have \re{lem:III} and \re{lem:IV}.

\re{lem:V}:
              We show that $f''(x)$ is increasing, which obviously follows of $f'''(x)>0$ for all $x\ge 1$. 
              We compute 
              \begin{equation}\label{f'''}
              f'''(x)= \frac{e^{6\hm(x)}p_1(x)+e^{2\hm(x)}\hm(x)^2p_2(x)+e^{4\hm(x)}p_3(x)-3\hm(x)^4              
              }{8(x^{3/2}\pi+e^{2\hm}x(\hm(x)-1))^{3}}
              \end{equation}
              with 
              \begin{align*}
              p_1(x) &:=  8-24\hm(x)+24\hm(x)^2-9\hm(x)^3+3\hm(x)^4,\\
              p_2(x) &:=  15-21\hm(x)+21\hm(x)^2-20\hm(x)^3+8\hm(x)^4,\\
              p_3(x) &:= -21\hm(x)+45\hm(x)^2-30\hm(x)^3+7\hm(x)^4+4\hm(x)^5-8\hm(x)^6.
              \end{align*}
              We note that the denominator of \re{f'''} is positive and nonzero for all $x\ge1$.
              Next, we will show that the numerator of \re{f'''} is positive. For this, we need
              \begin{align}
              p_1(x)>\hm(x)^4 \label{p_1_lower},\\
              p_2(x) > 4\hm(x)^4 \label{p_2_lower},\\
              p_3(x) > -8\hm(x)^6 \label{p_3_lower}.
              \end{align}
              We begin with the proof for \re{p_1_lower}. For $x\ge2>\frac{3}{\sqrt{2\pi}}$ we have $ \hm(x)-1 > 0,$ and we find 
              $$ 8 + 24 \hm(x)(\hm(x)-1)+\hm(x)^3(2\hm(x)-9)>0,$$
              which is equivalent to 
              $$ 8 -24 \hm(x) +24 \hm(x)^2 -9 \hm(x)^3 +2\hm(x)^4 >0,$$
              and can be rewritten as
              $ p_1(x) > \hm(x)^4.$

              Next, we prove \re{p_2_lower}. If $x>1>\sqrt{\frac{5}{2\pi}}$, then we find
             $$ 2\hm(x) > 5 \text{ and } \hm(x) -1 >0 $$
             which is equivalent to
              $$ 15 + 21\hm(x)(\hm(x)-1)+4\hm(x)^3(2\hm(x)-5) > 0$$
              and leads to $ p_2(x) >4\hm(x)^4$.
              
              Finally, we prove \re{p_3_lower}. For $x\ge 1$ we use $\hm(1)=\pi>3$, so that $\hm(x)^2>\hm(x)$. Thus,
              \begin{align*}
              p_3(x)+8\hm(x)^6 &> \ -21\hm(x)+45\hm(x)-30\hm(x)^3+21\hm(x)^3+4\hm(x)^5\\
              &> \ 24\hm(x)-9\hm(x)^3+36\hm(x)^3 
              = \ 24\hm(x)+27\hm(x)^3 
              > \ 0.             
              \end{align*}
              It remains to show that
              \begin{equation}\label{f>t>0}
              e^{6\hm(x)}\hm(x)^4+e^{2\hm(x)}\hm(x)^2 4\hm(x)^4-8e^{4\hm(x)}\hm(x)^6-3\hm(x)^4 > 0 .
              \end{equation}
              We do this by noticing that $e^{2\hm(x)}\hm(x)^2>3$ for $x\ge 1$, so that 
              \begin{equation*}
              e^{2\hm(x)}\hm(x)^2p_2(x)-3\hm(x)^4> 3\hm(x)^4 -3\hm(x)^4 = 0
              \end{equation*}
              and 
              \begin{equation}\label{argh}
              e^{6\hm(x)}\hm(x)^4-8\hm(x)^6 e^{4\hm(x)} = \hm(x)^4 e^{4\hm(x)}\left(e^{2\hm(x)}-\hm(x)^6\right) > 0 .
              \end{equation}
              We show \re{argh} by first noticing that $\eexp{2z}/z^6$ and $\hm(x)$ are increasing for $x\ge 1$. 
              Moreover, 
              $\frac{\eexp{2z}}{z^6} \rightarrow \infty,$
              and so 
             $\hm(x) \rightarrow \infty$
              for $x,z \rightarrow \infty$.
              Hence, there exists a $z_0 \in \R$ with 
              $$\frac{\eexp{2z}}{z^6} \ge 1 $$
              for $z\ge z_0$. There also exists an $x_0$ so that $\eexp{2\hm(x)}/\hm(x)^6>1$ for $x\ge x_0$, and when we have $\eexp{2\hm(3)}/\hm(3)^6> 1$,  we find 
              $$\frac{\eexp{2\hm(x)}}{\hm(x)^6}>1$$
              which is equivalent to
              $$\eexp{2\hm(x)}-\hm(x)^6>1$$
              for all $x\ge 3\ge x_0$.
	          This shows \re{f>t>0}, which indicates $f''(x)$ is increasing for $x\ge3$. This concludes the proof.
              \end{proof}
              
              Using \Cref{f:1}, we can  now prove the following theorem.
              
              \begin{thm}\label{OP:T}
              The function $\widehat{T}(n)$ is log-concave for all $n \ge 3$. In paticular, 
              $$\mathcal{C}(\widehat{T}(n)) \le \frac{2\hm^3-2\hm e^{4\hm}+5\hm^6 e^{2\hm}}{4n^2 e^{4\hm}(\hm -1)^2} \text{ for all } n\ge 3.$$
              \end{thm}
              \begin{proof}
              To show the log-concavity of $\widehat{T}(n)$, we have to show that
              $$\left( \log \left(\widehat{T}(n)\right)\right) ' <0.$$
              By \Cref{f:1} we have the preliminaries for \Cref{DP:lemma} are fulfilled by $\widehat{T}(n)$,  and may now use \re{OP:upper} to complete the proof.
              \end{proof}
              
              In order to establish the log-concavity of $\op$, we use the estimates
              \begin{equation*}
              \widehat{T}(n)\left(1-\frac{|\widehat{R}(n)|}{\widehat{T}(n)}\right) \le \op \le \widehat{T}(n)\left(1+\frac{|\widehat{R}(n)|}{\widehat{T}(n)}\right)
              \end{equation*}
              
              and the following lemma.
              
              \begin{lemma}\label{OP:yn}
              If $y_n:=\frac{|\widehat{R}(n)|}{\widehat{T}(n)}$, then $\log\left(\frac{(1-y_n)(1-y_n)}{(1-y_{n-1})(1-y_{n+1})}\right) \ge  \frac{9}{n\hm}e^{-2\hm /3}$ for $n\ge 2$.
              \end{lemma}
              \begin{proof}
              To begin, we observe that
              \begin{align}
              \left|\frac{e^{-\hm}}{8n\hm} + R_2(n,3)\right| &\le \left|\frac{e^{-\hm}}{8n\hm}\right| + \left|R_2(n,3)\right| \nonumber 
              \le \frac{e^{-\hm}}{8n\hm} + \frac{3^{5/2}}{n\sqrt{n}}\sinh \left(\frac{\hm}{3}\right) 
              \nonumber \\
              &\le   \frac{\left( \frac{e^{-2/3\hm}}{4}-1\right) e^{-\hm/3} + 3^{5/2} e^{\hm/3}}{2n\hm} \le \frac{3^{5/2}}{2n\hm} e^{\hm/3}. \label{OP:Rupp}
              \end{align}
              Using \re{OP:Tupp} and \re{OP:Rupp}, we find
              \begin{equation*}
              y_n \le \frac{3^{5/2}}{n\hm}e^{-2\hm /3} \text{ for } n\ge 1.
              \end{equation*}
              For $n\ge 2$, we have $0\le y_n< 1$ and
              \begin{align*}
              \log\left(\frac{(1-y_n)(1-y_n)}{(1-y_{n-1})(1-y_{n+1})}\right) 
              &= 2\log(1-y_n)-\log (1-y_{n-1}) -\log (1-y_{n+1})\\ 
              & \ge 2\log(1-y_n) 
              \ge \frac{-2y_n}{1-y_n} 
              \ge \frac{-2y_n}{1-y_2} 
              \ge -3y_n.
              \end{align*}
              Therefore, $\log\left(\frac{(1-y_n)(1-y_n)}{(1-y_{n-1})(1-y_{n+1})}\right) \ge - \frac{9}{n\hm}e^{-2\hm /3}$, as desired.
              \end{proof}
              
              The last lemma needed to complete the proof of \Cref{OP:opC} is the following.
              \begin{lemma}\label{misc:composition}
              	If the two functions $a(n),b(n)$ are log-concave for $n>n_0$, the following is true:
              	\begin{enumerate}[(i)]
              			\item They satisfy $\mathcal{C}(a(n)b(n))=\mathcal{C}(a(n))+\mathcal{C}(b(n))$.\label{ab:1}
              			\item The product $a(n)b(n)$ is also log-concave for $n>n_0$.\label{jetztaber} 
              	\end{enumerate}              
              \end{lemma}
              
              If we now combine \Cref{OP:T} with \Cref{OP:yn,misc:composition}, we can complete the proof of \Cref{OP:opC}, which implies \ref{main1} after a short calculation.
              \begin{thm}\label{OP:opC}
              The overpartition function $\op $ is log-concave for $n\ge4$. More precisely, for $n>8$ 
              $$\mathcal{C}\left(\op\right)\le \frac{2\hm^3-2\hm e^{4\hm}+5\hm^6 e^{2\hm}}{4n^2 e^{4\hm}(\hm -1)^2} + \frac{9}{n\hm}e^{-2\hm /3}.$$
              \end{thm}
              \begin{proof}
              To begin, we combine the fact $\op = \widehat{T}(n)\left(1+\frac{\widehat{R}(n)}{\widehat{T}(n)}\right)$ with \Cref{misc:composition} \re{jetztaber} to obtain
              \begin{align*}
              \mathcal{C}\left(\overline{p}(n)\right) = \mathcal{C}\left(\widehat{T}(n)\right) +\mathcal{C}\left(1+\frac{\widehat{R}(n)}{\widehat{T}(n)}\right).
              \end{align*}
              Next, we use \Cref{OP:T} and \Cref{OP:yn} to obtain   
              \begin{align}
              	\mathcal{C}\left(\overline{p}(n)\right) &\le \frac{2\hm^3-2\hm e^{4\hm}+5\hm^6 e^{2\hm}}{4n^2 e^{4\hm}(\hm -1)^2} + \frac{9}{n\hm}e^{-2\hm /3} \nonumber \\
              	&= \frac{1}{4n\hm}\left( 36e^{-2/3\hm}- \frac{\pi^2\left( 2-2e^{-4\hm}\hm^2 -5e^{-2\hm}\hm^5\right) }{(\hm -1)^2}\right) . \label{finalsteps}
              \end{align}            
              
              In order to show that \re{finalsteps} is negative, we only have to show that the inner of the brackets is negative for a certain $ n_o \in \N $. Therefore we can neglect the factor $ \frac{1}{4n\hm} $, because it is strictly positive. And now we proof the equivalent condition
              \begin{equation}
              	\underbrace{\frac{36(\hm -1)^2}{\pi^2}e^{-2/3\hm} + 2 e^{-4\hm}\hm^4 + 5e^{-2\hm}\hm^5}_{=:f(n)>0 \text{ for } n\in \N} < 2.
              \end{equation}
              The function $ f(n) $ is monotone decreasing and so it suffice to find a $ n_0\in \N $ so that $ f(n_0)<2 $. This is given for $ n_0\ge 4 $ and so it follows that 
              $$  \mathcal{C}(\op) \le 0 $$ for $ n \ge 4. $

%
%
%
%
%
%
%
%
%
%
%
%
              \end{proof}
              \Cref{main1} now follows for large $n$ from \Cref{OP:opC} and for small $n$ from \Cref{t1}.    
%
%